\documentclass[12pt]{amsart}
\usepackage{latexsym,amssymb,amscd}
\usepackage{graphicx}
\def\B'c{{\mathcal{B'}}}
\def\U'c{{\mathcal{U'}}}

\def\xb{{\bold x}}

\def\opn#1#2{\def#1{\operatorname{#2}}} 
\opn\chara{char}
\opn\length{\ell}
\opn\cd{cd}
\opn\projdim{pd}
\opn\injdim{inj\,dim}
\opn\ini{in}
\opn\rank{rank}
\opn\depth{depth}
\opn\height{ht}
\opn\bigheight{bight}
\opn\embdim{emb\,dim}
\opn\codim{codim}

\opn\Tr{Tr}
\opn\bigrank{big\,rank}
\opn\superheight{superheight}\opn\lcm{lcm}
\opn\trdeg{tr\,deg}%
\opn\reg{reg}
\opn\lreg{lreg}
\opn\set{set}
\opn\supp{Supp}
\opn\shad{Shad}
\opn\indeg{indeg}
\opn\lex{lex}
\opn\div{div}
\opn\Div{Div}
\opn\cl{cl}
\opn\Cl{Cl}
\opn\Spec{Spec}
\opn\Supp{Supp}
\opn\supp{supp}
\opn\Sing{Sing}
\opn\Ass{Ass}
\opn\Ann{Ann}
\opn\Rad{Rad}
\opn\Soc{Soc}
\opn\Ker{Ker}
\opn\Coker{Coker}
\opn\Im{Im}
\opn\Hom{Hom}
\opn\Tor{Tor}
\opn\Ext{Ext}
\opn\End{End}
\opn\Aut{Aut}
\opn\id{id}

\opn\nat{nat}
\opn\GL{GL}
\opn\SL{SL}
\opn\mod{mod}
\opn\ord{ord}
\opn\ara{ara}
\opn\aff{aff}
\opn\con{conv}
\opn\relint{relint}
\opn\st{st}
\opn\lk{lk}
\opn\cn{cn}
\opn\core{core}
\opn\vol{vol}
\opn\gr{gr}

\def\pot#1#2{#1[\kern-0.28ex[#2]\kern-0.28ex]}

\opn\dirlim{\underrightarrow{\lim}}
\opn\invlim{\underleftarrow{\lim}}
\def\pnt{{\raise0.5mm\hbox{\large\bf.}}}

\def\Implies{\ifmmode\Longrightarrow \else
	\unskip${}\Longrightarrow{}$\ignorespaces\fi}
\def\implies{\ifmmode\Rightarrow \else
	\unskip${}\Rightarrow{}$\ignorespaces\fi}
\def\iff{\ifmmode\Longleftrightarrow \else
	\unskip${}\Longleftrightarrow{}$\ignorespaces\fi}

\let\:=\colon
\newtheorem{Theorem}{Theorem}[section]
\newtheorem{Lemma}[Theorem]{Lemma}
\newtheorem{Corollary}[Theorem]{Corollary}
\newtheorem{Proposition}[Theorem]{Proposition}
\newtheorem{Remark}[Theorem]{Remark}

\let\epsilon=\varepsilon
\let\phi=\varphi
\let\kappa=\varkappa
\numberwithin{equation}{section}

\title{On the closed neighborhood ideal of the square of broom and double broom graphs}
\author{Anda Olteanu \and Oana Olteanu}

\address{Faculty of Marine Engineering, Romanian Naval Academy ``Mircea cel B\u atr\^an", Fulgerului Street, no. 1,
	900218 Constanta, Romania,} \email{olteanuandageorgiana@gmail.com}
\address{Faculty of Applied Sciences,
	National University of Science and Technology Politehnica Bucharest, 
	Splaiul Independen\c tei, No.
	313, 060042, Bucharest, Romania}\email{olteanuoanastefania@gmail.com} 
\begin{document}
		\begin{abstract} We consider the closed neighborhood ideal of the square of two classes of graphs: broom and double broom graphs. We study their invariants such as the projective dimension, the Castelnuovo--Mumford regularity and we characterize the Cohen--Macaulay ideals. 		
	
		\textbf{Keywords}: squarefree monomial ideals, Castelnuovo--Mumford regularity, projective dimension, square of a graph\\ 
		
		\textbf{MSC(2020)}: Primary 13D05 Secondary 05E40
	\end{abstract}
	\maketitle

	\section{Introduction}
	
	Classes of squarefree monomial ideals were intensively studied being an important connection between two main areas in mathematics: commutative algebra and combinatorics. Many algebraic invariants of the squarefree monomial ideals are given in terms of properties of the combinatorial associated objects (graphs, simplicial complexes etc.). 
	
	In 2020 L. Sharifan and S. Moradi \cite{SM} introduced a new class of squarefree monomial ideals, namely the closed neighborhood ideal associated to a graph. For a finite simple graph $G$ with vertex set $V(G)$, the closed neighborhood ideal $NI(G)\subset K[V(G)]$, where $K$ is a field, is the squarefree monomial ideal with a monomial generating set corresponding to the closed neighborhoods of the vertices. Since then, many authors studied several properties of $NI(G)$ for various classes of graphs. For instance, L. Sharifan and S. Moradi in \cite{SM} analyzed path graphs, generalized star graphs and $m-$book graphs. They related invariants such as the height of $NI(G)$, the projective dimension and the Castelnuovo--Mumford regularity of $K[V(G)]/NI(G)$ with notions as minimal dominating set or the matching number of $G$. In \cite{HSW} J. Honeycutt and K. Sather-Wagstaff characterized classes of trees such that $NI(G)$ are Cohen-Macaulay and S. Chakraborty, A.P. Joseph, A. Roy, A. Singh in \cite{CJRS} proved that for the case of forests, the Castelnuovo--Mumford regularity of $NI(G)$ equals the matching number, proving the conjecture given by L. Sharifan and S. Moradi \cite{SM}.
	
	Other important results were given in \cite {NBR}, \cite {NQ} and \cite{NQBM} where the authors studied the normally-torsion free property and the (strong) persistence property of the closed neighborhood ideals of some classes of graphs. Moreover, the minimal free resolution of closed neighborhood ideals in the framework of Barile-Macchia resolutions was given in \cite{JRS} where they proved that the Barile-Macchia resolution of $NI(G)$ is minimal, in the case of trees. The associated primes of the second power of closed neighborhood ideals were studied in \cite{HV} by H.T.T. Hien and T. Vu. Worth mentioning the paper of D. Jaramillo-Velez, H.H. L$\acute{o}$pez, R. San-Jos$\acute{e}$ \cite{JLR} where the authors estimate the $v$-number of closed neighborhood ideals in terms of minimal dominating sets and private neighbors. 
	
	In this paper we consider the closed neighborhood ideal of the square of broom and double broom graph. We are interested in studying algebraic invariants such as the height, the projective dimension and the Castelnuovo--Mumford regularity. All these invariants allow us to characterize the Cohen--Macaulay property. The paper is structured as follows: we begin by briefly recalling some useful notions concerning the two classes of considered graphs and the closed neighborhood ideal. In the next section we pay attention to the first class of graphs, the broom graphs, and we describe in Theorem \ref{projdim+reg} the projective dimension and the Castelnuovo--Mumford regularity of the closed neighborhood ideal of the square of broom graphs. The main idea is to use the mapping cone technique given in \cite{S}. Next we determine in Theorem \ref{height} the height and we characterize the Cohen--Macaulay ideals in Corollary \ref{CM broom}. The last section treats the same invariants for the closed neighborhood ideal of the square of double broom graphs. The projective dimension and the Castelnuovo--Mumford regularity of the closed neighborhood ideal of the square of double broom graphs are computed in Theorem \ref{projdim+reg double-broom}, while the height and the Cohen--Macaulay characterization are given in Theorem \ref{height double} and Corollary \ref{CM double}.
	
	\section{Preliminaries}
	
	Let $n, m ,s\geq 2$ be natural numbers. Denote by $G=B(n,m)$ the broom graph, that is the graph on the set of vertices $V(G)=\{x_1,\ldots,x_n,y_1,\ldots,y_m\}$ and the set of edges $$E(G)=\{\{x_i,y_1\}:1\leq i\leq n\}\cup\{\{y_j,y_{j+1}\}:1\leq j\leq m-1\}.$$
	
	\begin{figure}[h]
		\begin{center}
			\includegraphics[width=8cm]{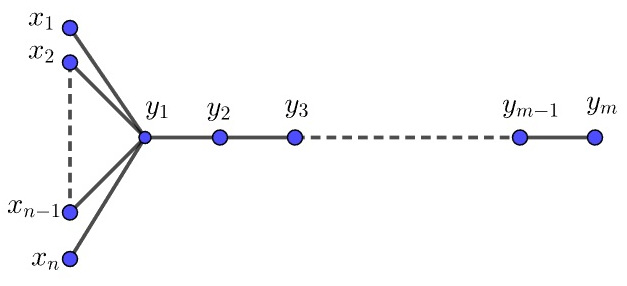}
			\caption{The broom graph}
		\end{center}
	\end{figure}
	The square of broom graph, $G^2=B^2(n,m)$, is the graph on the vertex set $V(G)$ and the set of edges:
	$$E(G^2)=E(G)\cup\{\{x_i,x_j\}:1\leq i< j\leq n\}\cup\{\{x_i,y_2\}:1\leq i\leq n\}\cup$$
	$$\cup \{\{y_j,y_{j+2}\}:1\leq j\leq m-2\}$$
	
	Moreover, we denote by $G_1=B(n,m,s)$ the double broom graph, that is the graph on the set of vertices $V(G_1)=\{x_1,\ldots,x_n,y_1,\ldots,y_m,z_1,\ldots,z_s\}$ and the set of edges 
	$$E(G_1)=\{\{x_i,y_1\}:1\leq i\leq n\}\cup\{\{y_j,y_{j+1}\}:1\leq j\leq m-1\}\cup\{\{y_m,z_r\}:1\leq r\leq s\}.$$
	\begin{figure}[h]
		\begin{center}
			\includegraphics[width=9cm]{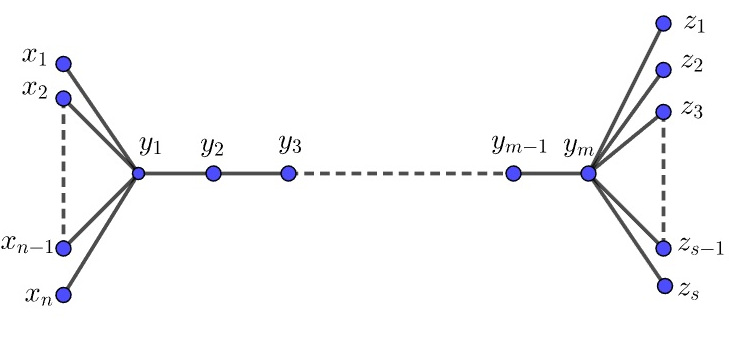}
			\caption{The double broom graph}
		\end{center}
	\end{figure}
	
	Similarly, we consider the square of double broom graph, $G_1^2=B^2(n,m,s)$, that is the graph on the vertex set $V(G_1)$ and the set of edges:
	$$E(G_1^2)=E(G_1)\cup\{\{x_i,x_j\}:1\leq i< j\leq n\}\cup\{\{x_i,y_2\}:1\leq i\leq n\}\cup$$
	$$\cup \{\{y_j,y_{j+2}\}:1\leq j\leq m-2\}\cup\{\{y_{m-1},z_r\}:1\leq r\leq s\}\cup\{\{z_r,z_h\}:1\leq r< h\leq s\}.$$ 
	
	We recall that for an arbitrary simple graph $G$ the \textit{neighborhood} of the vertex $x\in V(G)$ is the set
	$$N(x)=\{y\in V(G):\{x,y\}\in E(G)\}$$
	and the \textit{closed neighborhood} of the vertex $x$ is $N[x]=N(x)\cup\{x\}$. 
	L. Sharifan and S. Moradi introduced in \cite{SM} the notion of the \textit{closed neighborhood ideal} of a graph, denoted by $NI(G)$, which is the squarefree monomial ideal of the polynomial ring $K[x_i:i\in V(G)]$ over a field $K$
	$$NI(G)=(\xb_{N[u]}=\prod\limits_{j\in N[u]}x_j: u\in V(G)).$$
	
	Informations about the algebraic notions we are interested in can be found for instance in \cite{V}. In the sequel, let $S=K[x_1,\ldots,x_n]$ be the polynomial ring in $n$ variables over a field $K$. We recall that for an ideal $I\subset S$, if we consider the minimal free resolution of $S/I$ as an $S-$module
	$$0\rightarrow \bigoplus\limits_{j} S(-j)^{\beta_{qj}}\rightarrow\ldots \rightarrow \bigoplus\limits_{j} S(-j)^{\beta_{1j}}\rightarrow S\rightarrow S/I\rightarrow 0,$$
	then the \textit{projective dimension} of $S/I$ is 
	$$\projdim(S/I)=\max\{i:\beta_{ij}\neq 0\}$$
	and the \textit{Castelnuovo--Mumford regularity} is
	$$\reg(S/I)=\max\{j-i:\beta_{ij}\neq 0\}.$$
	
	A very useful tool to compute the projective dimension and the Castelnuovo--Mumford regularity is the mapping cone technique. We recall the following results given by L. Sharifan in \cite {S}:
	
	\begin{Theorem}\cite[Theorem 2.4]{S}
		Let $I$ be a graded ideal of $S$ and $f$ be a homogeneous polynomial of degree $d$ which does not belong to $I$ then we have the following graded short exact sequence
		$$0\rightarrow\frac{S}{(I:f)}(-d)\rightarrow\frac{S}{I}\rightarrow\frac{S}{I+(f)}\rightarrow 0.$$
		Assuming that the minimal free resolution of the modules $S/(I:f)$ and $S/I$ are already known. Then the minimal free resolution of $S/(I+(f))$ is obtained by the mapping cone provided that $f=h_1h_2$ where $h_1$ and $h_2$ are homogeneous polynomials, $\deg(h_2)>0$ and $(I:f)=(I:h_1)$ and in this case:
		\begin{itemize}
			\item [(a)] $\beta_{ij}\left(\frac{S}{I+(f)}\right)=\beta_{ij}\left(\frac{S}{I}\right)+\beta_{i-1j-d}\left(\frac{S}{(I:f)}\right)$,
			\item [(b)] $\reg\left(\frac{S}{I+(f)}\right)=\max\{\reg\left(\frac{S}{I}\right),\reg\left(\frac{S}{(I:f)}\right)+d-1\}$,
			\item [(c)] $\projdim\left(\frac{S}{I+(f)}\right)=\max\{\projdim\left(\frac{S}{I}\right),\projdim\left(\frac{S}{(I:f)}\right)+1\}$.
		\end{itemize} 
	\end{Theorem}
	
	In particular, it results that:
	
	\begin{Corollary}\cite[Corollary 2.6]{S} If $I$ is a monomial ideal of $S$ and $u$ is a monomial which does not belong to $I$, then the minimal free resolution of $S/(I+(u))$ is given by the mapping cone technique provided that there is $x_i\in\supp(u)$ such that for all $v\in G(I)$ $\deg_{x_i} (u) > \deg_{x_i}(v)$. Here, $\supp(u)=\{x_i:x_i\mid u\}$.
	\end{Corollary}
	
	In \cite{OO} the mapping cone technique was used to compute the projective dimension and the Castelnuovo--Mumford regularity of the closed neighborhood ideal of the square of path graph $NI(P_n^2)$:
	
	\begin{Theorem}\cite[Theorem 1.9]{OO}\label{NI(P^2_n)}
		For $n=6p+d$, we have
		$$\reg\left(\frac{S}{NI(P^2_{n})}\right)=\left\{\begin{matrix}
			4p,&d\in\{0,1\}\\
			4p+1,&d=2\\
			4p+2,&d\in\{3,4\}\\
			4p+3,&d=5
		\end{matrix}\right.$$
		and $$\projdim\left(\frac{S}{NI(P^2_{n})}\right)=\left\{\begin{matrix}
			2p,&d=0\\
			2p+1,&d\in\{1,2,3\}\\
			2p+2,&d\in\{4,5\}
		\end{matrix}\right.$$
	\end{Theorem} 	
	
	In the sequel, we will pay attention to the closed neighborhood ideal of two classes of graphs: broom and double broom.

	\section{The closed neighborhood ideal of the square of broom graph}
	Let $G=B(n,m)$ with $n,m\geq 2$ be the broom graph and
	$$NI(G^2)=(x_1\cdots x_ny_1y_2)+(y_jy_{j+1}y_{j+2}y_{j+3}y_{j+4}:1\leq j\leq m-5)+(y_{m-2}y_{m-1}y_m)$$ 
	the closed neighborhood ideal of the square of broom graph $G^2$, where $NI(G^2)\subset S=K[x_1,\ldots,x_n,y_1,\ldots,y_m]$.
	
	We distinguish the following particular cases:
	\begin{itemize}
		\item for $m=2$ the closed neighborhood ideal of the square of broom graph $G^2$ is $NI(G^2)=(x_1\cdots x_ny_1y_2)$;
		\item for $m=3$ we obtain $NI(G^2)=(x_1\cdots x_ny_1y_2, y_1y_2y_3)$;
		\item for $m=4$ it results that $NI(G^2)=(x_1\cdots x_ny_1y_2,y_2y_3y_4)$;
		\item for $m=5$ we have $NI(G^2)=(x_1\cdots x_ny_1y_2,y_3y_4y_5)$.
	\end{itemize}
	
	For these particular cases we obtain the following formulae for the invariants we are interested in:
	\begin{Proposition}\label{particular cases 1}
		Let $NI(G^2)\subset S$ be the closed neighborhood ideal of the square of broom graph $G^2$, where $G=B(n,m)$, with $2\leq m\leq 5$. Then:
		\begin{itemize}
			\item for $m=2$ one has that $\height NI(G^2)=1$, $\projdim(S/NI(G^2))=1$ and $\reg(S/NI(G^2))=n+1$;
			\item for $m=3$ one has that $\height NI(G^2)=1$, $\projdim(S/NI(G^2))=2$ and $\reg(S/NI(G^2))=n+1$; 
			\item for $m=4$ one has that $\height NI(G^2)=1$, $\projdim(S/NI(G^2))=2$ and $\reg(S/NI(G^2))=n+2$;
			\item for $m=5$ one has that $\height NI(G^2)=2$, $\projdim(S/NI(G^2))=2$ and $\reg(S/NI(G^2))=n+3$.
		\end{itemize}
	\end{Proposition}
	
	\begin{proof}
		For the cases $m\in\{2,3,4\}$ we have that $(y_2)$ is a minimal prime ideal, hence $\height NI(G^2)=1$, while for $m=5$ a minimal prime ideal is $(y_1,y_3)$, therefore $\height NI(G^2)=2$. 
		
		For $m=2$, it is clear that $\projdim(S/NI(G^2))=1$ and $\reg(S/NI(G^2))=n+1$.
		
		For $m=3$ we have that $NI(G^2)=(x_1\cdots x_ny_1y_2, y_1y_2y_3)=y_1y_2(x_1\cdots x_n,y_3)$ and $(x_1\cdots x_n,y_3)$ is a complete intersection ideal, therefore $\reg(S/NI(G^2))=\reg(NI(G^2))-1=n+1$, where $\reg(NI(G^2))=2+[\deg(x_1\cdots x_n)+\deg(y_3)+2-1]=n+2$. 
		
		We use the same argument of complete intersection ideal for computing the Castelnuovo--Mumford regularity for the cases $m=3$ and $m=4$.
		
		The formulae for projective dimension are clear and also can be determined using Singular. 
	\end{proof}
	
	We also remark that, for $m\geq 6$, the closed neighborhood ideal of the square of broom graph $G^2$ can be written as
	$$NI(G^2)=(x_1\cdots x_ny_1y_2)+(y_{m-2}y_{m-1}y_m)+I_5(P_{m-1})$$ 
	where by $I_t(P_m)$ we mean the ideal of the paths of length $t$ of the line graph $P_{\{1,\ldots,m\}}$, $I_t(P_m)\subset S'=K[y_1,\ldots,y_m]$. But the projective dimension and the Castelnuovo--Mumford regularity of path ideals of lines were given in \cite[Theorem 4.1]{HeT} and in \cite[Corollary 4.14]{AF}:
	
	\begin{Corollary}
		Let $n,t,p,d$ be integers such that $n\geq 2$, $2\leq t\leq n$, $n=p(t+1)+d$, where $p\geq 0$ and $0\leq d\leq t$. Then
		\begin{itemize}
			\item \cite[Theorem 4.1]{HeT}the projective dimension of the path ideal of a path $P_n$ is given by 
			$$\projdim(S/I_t(P_n))=\left\{\begin{matrix}
				2p,&d\neq t\\
				2p+1,&d=t
			\end{matrix}\right.$$
			\item \cite[Corollary 4.14]{AF}the regularity of the path ideal of a path $P_n$ is given by 
			$$\reg(S/I_t(P_n))=\left\{\begin{matrix}
				p(t-1),&d\neq t\\
				(p+1)(t-1),&d=t
			\end{matrix}\right..$$
		\end{itemize}
	\end{Corollary}
	
	In our case, for $I_5(P_{m-1})$ which appears in the written of the closed neighborhood ideal of the square of broom graph, we obtain the following formulae:
	\begin{Remark}	
		For $t=5$ and $m-1=6p'+d'$, we obtain that
		\begin{itemize}
			\item the projective dimension of the path ideal of a path $P_{m-1}$ is given by 
			$$\projdim(S'/I_5(P_{m-1}))=\left\{\begin{matrix}
				2p',&d'\neq 5\\
				2p'+1,&d'=5
			\end{matrix}\right.$$
			\item the regularity of the path ideal of a path $P_{m-1}$ is given by 
			$$\reg(S'/I_5(P_{m-1}))=\left\{\begin{matrix}
				4p',&d'\neq 5\\
				4(p'+1),&d'=5
			\end{matrix}\right..$$
		\end{itemize}
	\end{Remark}
	
	In order to compute the projective dimension and the Castelnuovo--Mumford regularity of $S/NI(G^2)$ we consider the short exact sequence
	$$0\rightarrow \frac{S}{(I_5(P_{m-1})+(y_{m-2}y_{m-1}y_m)):(x_1\cdots x_ny_1y_2)}(-n-2)\rightarrow$$
	$$\rightarrow \frac{S}{I_5(P_{m-1})+(y_{m-2}y_{m-1}y_m)}\rightarrow \frac{S}{NI(G^2)}\rightarrow 0$$
	and by the mapping cone technique \cite[Theorem 2.4]{S} we obtain that
	$$\reg\frac{S}{NI(G^2)}=$$
	$$=\max\left\{\reg \frac{S}{(I_5(P_{m-1})+(y_{m-2}y_{m-1}y_m)):(x_1\cdots x_ny_1y_2)}+n+1,\reg \frac{S}{I_5(P_{m-1})+(y_{m-2}y_{m-1}y_m)}\right\}$$
	and 
	$$\projdim\frac{S}{NI(G^2)}=$$
	$$=\max\left\{\projdim \frac{S}{(I_5(P_{m-1})+(y_{m-2}y_{m-1}y_m)):(x_1\cdots x_ny_1y_2)}+1,\projdim \frac{S}{I_5(P_{m-1})+(y_{m-2}y_{m-1}y_m)}\right\}.$$

	We start by computing the projective dimension and the Castelnuovo--Mumford regularity of $\frac{S}{(I_5(P_{m-1})+(y_{m-2}y_{m-1}y_m)):(x_1\cdots x_ny_1y_2)}$.
	\begin{Lemma}\label{quotient}
		If $m-2=6p''+d''$, with $0\leq d''\leq 5$, then one has
		$$\reg \frac{S}{(I_5(P_{m-1})+(y_{m-2}y_{m-1}y_m)):(x_1\cdots x_ny_1y_2)}=\left\{\begin{matrix}
			4p'',&d''\in\{0,1\}\\
			4p''+1,&d''=2\\
			4p''+2,&d''\in\{3,4\}\\
			4p''+3,&d''=5
		\end{matrix}\right.$$
		and
		$$\projdim\frac{S}{(I_5(P_{m-1})+(y_{m-2}y_{m-1}y_m)):(x_1\cdots x_ny_1y_2)}=\left\{\begin{matrix}
			2p'',&d''=0\\
			2p''+1,&d''\in\{1,2,3\}\\
			2p''+2,&d''\in\{4,5\}
		\end{matrix}\right..$$
	\end{Lemma}
	
	\begin{proof}
		It is easy to observe that:
		$$(I_5(P_{m-1})+(y_{m-2}y_{m-1}y_m)):(x_1\cdots x_ny_1y_2)=$$
		$$=I_5(P_{m-1}):(x_1\cdots x_ny_1y_2)+(y_{m-2}y_{m-1}y_m):(x_1\cdots x_ny_1y_2)=$$
		$$=I_5(P_{m-1}):(y_1y_2)+(y_{m-2}y_{m-1}y_m)=$$
		$$=(y_3y_4y_5,y_4y_5y_6y_7y_8,\ldots,y_{m-5}y_{m-4}y_{m-3}y_{m-2}y_{m-1},y_{m-4}y_{m-3}y_{m-2}y_{m-1}y_{m})+(y_{m-2}y_{m-1}y_m)$$
		$$=(y_3y_4y_5,y_4y_5y_6y_7y_8,\ldots,y_{m-5}y_{m-4}y_{m-3}y_{m-2}y_{m-1},y_{m-2}y_{m-1}y_m)=NI(P^2_{\{3,\ldots,m\}})$$
		where $NI(P^2_{\{3,\ldots,m\}})$ is the closed neighborhood ideal on the square of path graph on the vertex set $\{3,\ldots,m\}$. By Theorem \ref{NI(P^2_n)} we obtain that if $m-2=6p''+d''$, one has
		$$\reg \frac{S'}{NI(P^2_{\{3,\ldots,m\}})}=\left\{\begin{matrix}
			4p'',&d''\in\{0,1\}\\
			4p''+1,&d''=2\\
			4p''+2,&d''\in\{3,4\}\\
			4p''+3,&d''=5
		\end{matrix}\right.$$
		and
		$$\projdim\frac{S'}{NI(P^2_{\{3,\ldots,m\}})}=\left\{\begin{matrix}
			2p'',&d''=0\\
			2p''+1,&d''\in\{1,2,3\}\\
			2p''+2,&d''\in\{4,5\}
		\end{matrix}\right.$$
		Therefore
		$$\reg \frac{S}{(I_5(P_{m-1})+(y_{m-2}y_{m-1}y_m)):(x_1\cdots x_ny_1y_2)}=\left\{\begin{matrix}
			4p'',&d''\in\{0,1\}\\
			4p''+1,&d''=2\\
			4p''+2,&d''\in\{3,4\}\\
			4p''+3,&d''=5
		\end{matrix}\right.$$
		and
		$$\projdim\frac{S}{(I_5(P_{m-1})+(y_{m-2}y_{m-1}y_m)):(x_1\cdots x_ny_1y_2)}=\left\{\begin{matrix}
			2p'',&d''=0\\
			2p''+1,&d''\in\{1,2,3\}\\
			2p''+2,&d''\in\{4,5\}
		\end{matrix}\right..$$
	\end{proof}
	
	For the invariants of $\frac{S}{I_5(P_{m-1})+(y_{m-2}y_{m-1}y_m)}$, we need to recall the following result from \cite{OO}:
	
	\begin{Lemma}\label{reg1}\cite[Lemma 1.10]{OO}
		For $n=6p+d$, with $p\geq 1$, one has
		$$\reg\left(\frac{K[x_1,\ldots,x_n]}{I_5(P_{\{2,\ldots,n-1\}})+(x_1x_2x_3)}\right)=\left\{\begin{matrix}
			4p-2,&d=0\\
			4p,&d\in\{1,2,3\}\\
			4p+2,&d\in\{4,5\}
		\end{matrix}\right.$$
		and $$\projdim\left(\frac{K[x_1,\ldots,x_n]}{I_5(P_{\{2,\ldots,n-1\}})+(x_1x_2x_3)}\right)=\left\{\begin{matrix}
			2p-1&d=0\\
			2p,&d\in\{1,2,3\}\\
			2p+1,&d\in\{4,5\}
		\end{matrix}\right..$$
	\end{Lemma}
	
	\begin{Remark}\label{sum}\rm
		In our case, if we relabel the variables $y_i$ by $\alpha_{m-i+1}$, for $1\leq i\leq m$, we obtain that
		$$I_5(P_{m-1})+(y_{m-2}y_{m-1}y_m)=(y_jy_{j+1}y_{j+2}y_{j+3}y_{j+4}:1\leq j\leq m-5)+(y_{m-2}y_{m-1}y_m)=$$
		$$=(\alpha_{m-j-3}\alpha_{m-j-2}\alpha_{m-j-1}\alpha_{m-j}\alpha_{m-j+1}:1\leq j\leq m-5)+(\alpha_{1}\alpha_{2}\alpha_3)=$$
		$$=I_5(P_{\{2,\ldots m\}})+(\alpha_{1}\alpha_{2}\alpha_3).$$
		Therefore, if $m+1=6p'+d'$, with $p'\geq 1$, one has
		$$\reg\left(\frac{K[y_1,\ldots,y_m]}{I_5(P_{\{2,\ldots,m\}})+(y_{m-2}y_{m-1}y_m)}\right)=$$
		$$=\reg\left(\frac{K[\alpha_1,\ldots,\alpha_m]}{I_5(P_{\{2,\ldots,m\}})+(\alpha_1\alpha_2\alpha_3)}\right)=\left\{\begin{matrix}
			4p'-2,&d'=0\\
			4p',&d'\in\{1,2,3\}\\
			4p'+2,&d'\in\{4,5\}
		\end{matrix}\right.$$
		and $$\projdim\left(\frac{K[y_1,\ldots,y_m]}{I_5(P_{\{2,\ldots,m\}})+(y_{m-2}y_{m-1}y_m)}\right)=$$
		$$=\projdim\left(\frac{K[\alpha_1,\ldots,\alpha_m]}{I_5(P_{\{2,\ldots,m\}})+(\alpha_1\alpha_2\alpha_3)}\right)=\left\{\begin{matrix}
			2p'-1&d'=0\\
			2p',&d'\in\{1,2,3\}\\
			2p'+1,&d'\in\{4,5\}
		\end{matrix}\right.$$
		We recall that the projective dimension and the Castelnuovo--Mumford regularity for the first module $\frac{S}{(I_5(P_{m-1})+(y_{m-2}y_{m-1}y_m)):(x_1\cdots x_ny_1y_2)}$ were given when $m-2=6p''+d''$. Therefore, for $m-2=6p''+d''$ we obtain that
		$$\reg\left(\frac{K[y_1,\ldots,y_m]}{I_5(P_{\{2,\ldots,m\}})+(y_{m-2}y_{m-1}y_m)}\right)=\left\{\begin{matrix}
			4p'',&d''=0\\
			4p''+2,&d''\in\{1,2,3\}\\
			4p''+4,&d''\in\{4,5\}
		\end{matrix}\right.$$
		and $$\projdim\left(\frac{K[y_1,\ldots,y_m]}{I_5(P_{\{2,\ldots,m\}})+(y_{m-2}y_{m-1}y_m)}\right)=\left\{\begin{matrix}
			2p''&d''\in\{0,3\}\\
			2p''+1,&d''\in\{1,2\}\\
			2p''+2,&d''\in\{4,5\}
		\end{matrix}\right..$$
	\end{Remark}
	
	Now we are able to compute the Castelnuovo--Mumford regularity and the projective dimension of $S/NI(G^2)$, where $G$ is the broom graph.
	
	\begin{Theorem}\label{projdim+reg}
		Let $G=B(n,m)$ the broom graph on the set of vertices $V(G)=\{x_1,\ldots,x_n,y_1,\ldots,y_m\}$ and $NI(G^2)\subset S=K[V(G)]$ the closed neighborhood ideal of the square of $G$. Then if $m-2=6p''+d''$, with $p''\geq 0$ and $0\leq d''\leq 5$, then:
		$$\reg \frac{S}{NI(G^2)}=\left\{\begin{matrix}
			4p''+n+1,&d''\in\{0,1\}\\
			4p''+n+2,&d''=2\\
			4p''+n+3,&d''\in\{3,4\}\\
			4p''+n+4,&d''=5
		\end{matrix}\right.$$
		and $$\projdim\frac{S}{NI(G^2)}=\left\{\begin{matrix}
			2p''+1&d''=0\\
			2p''+2,&d''\in\{1,2,3\}\\
			2p''+3,&d''\in\{4,5\}
		\end{matrix}\right..$$
\end{Theorem}

\begin{proof}
	Firstly, we assume that $m-2=6p''+d''$, with $p''\geq 1$ and $0\leq d''\leq 5$. The proof follows by the mapping cone technique 
	$$\reg\frac{S}{NI(G^2)}=$$
	$$=\max\left\{\reg \frac{S}{(I_5(P_{m-1})+(y_{m-2}y_{m-1}y_m)):(x_1\cdots x_ny_1y_2)}+n+1,\reg \frac{S}{I_5(P_{m-1})+(y_{m-2}y_{m-1}y_m)}\right\}$$
	and 
	$$\projdim\frac{S}{NI(G^2)}=$$
	$$=\max\left\{\projdim \frac{S}{(I_5(P_{m-1})+(y_{m-2}y_{m-1}y_m)):(x_1\cdots x_ny_1y_2)}+1,\projdim \frac{S}{I_5(P_{m-1})+(y_{m-2}y_{m-1}y_m)}\right\}$$
	and by Lemma \ref{quotient} and Remark \ref{sum}.
	
	It remains to study $m-2=6p''+d''$ with $p''=0$ and $0\leq d''\leq 5$. The cases $m-2=6p''+d''$ with $p''=0$ and $0\leq d''\leq 3$ were given in Proposition \ref{particular cases 1}.
	
	Assume that $m-2=6p''+d''$ with $p''=0$ and $d''=4$, that is $m=6$. In this case we have that 
	$$NI(G^2)=(x_1\cdots x_ny_1y_2,y_1y_2y_3y_4y_5, y_4y_5y_6)=(x_1\cdots x_ny_1y_2)+I_5(P_{5})+(y_4y_5y_6)$$ 
	and by the mapping cone technique \cite[Theorem 2.4]{S} we obtain that
	$$\reg\frac{S}{NI(G^2)}=\max\left\{\reg \frac{S}{(I_5(P_{5})+(y_{4}y_{5}y_6)):(x_1\cdots x_ny_1y_2)}+n+1,\reg \frac{S}{I_5(P_{5})+(y_{4}y_{5}y_6)}\right\}$$
	and 
	$$\projdim\frac{S}{NI(G^2)}=\max\left\{\projdim \frac{S}{(I_5(P_{5})+(y_{4}y_{5}y_6)):(x_1\cdots x_ny_1y_2)}+1,\projdim \frac{S}{I_5(P_{5})+(y_{4}y_{5}y_6)}\right\}.$$
	By Lemma \ref{quotient}, for $m-2=6p''+d''$ with $p''=0$ and $d''=4$ we obtain that 
	$$\reg \frac{S}{(I_5(P_{5})+(y_{4}y_{5}y_6)):(x_1\cdots x_ny_1y_2)}=2$$
	and
	$$\projdim\frac{S}{(I_5(P_{5})+(y_{4}y_{5}y_6)):(x_1\cdots x_ny_1y_2)}=2.$$
	
	We use Singular \cite{DGPS} to compute 
	$$\reg \frac{S}{I_5(P_{5})+(y_{4}y_{5}y_6)}= 4,\ \projdim \frac{S}{I_5(P_{5})+(y_{4}y_{5}y_6)}=2.$$
	Therefore, for $m=6$, we obtain that
	$$\reg\frac{S}{NI(G^2)}=\max\left\{2+n+1,4\right\}=n+3$$
	and 
	$$\projdim\frac{S}{NI(G^2)}=\max\left\{2+1,2\right\}=3.$$
	
	For the last case, $m-2=6p''+d''$ with $p''=0$ and $d''=5$, that is $m=7$, by Lemma \ref{quotient} we obtain that 
	$$\reg \frac{S}{(I_5(P_{6})+(y_{5}y_{6}y_7)):(x_1\cdots x_ny_1y_2)}=3$$
	and
	$$\projdim\frac{S}{(I_5(P_{6})+(y_{5}y_{6}y_7)):(x_1\cdots x_ny_1y_2)}=2.$$
	
	Once again we use Singular to determine that 
	$$\reg \frac{S}{I_5(P_{6})+(y_{5}y_{6}y_7)}= 4,\ \projdim \frac{S}{I_5(P_{6})+(y_{5}y_{6}y_7)}=2.$$
	Therefore, for $m=6$, we obtain that
	$$\reg\frac{S}{NI(G^2)}=\max\left\{3+n+1,4\right\}=n+4$$
	and 
	$$\projdim\frac{S}{NI(G^2)}=\max\left\{2+1,2\right\}=3.$$
\end{proof}

Another important invariant which we are interested in computing is the height of the closed neighborhood ideal of the square of broom graph.

\begin{Theorem}\label{height}
	Let $G=B(n,m)$ the broom graph on the set of vertices $V(G)=\{x_1,\ldots,x_n,y_1,\ldots,y_m\}$ and $NI(G^2)\subset S=K[V(G)]$ the closed neighborhood ideal of the square of $G$. Then $\height NI(G^2)=\left\lceil\frac{m+1}{5}\right\rceil.$
\end{Theorem}

\begin{proof} The particular cases $2\leq m\leq 5$ were treated in Proposition \ref{particular cases 1}.
	
	Let $m\geq 6$ and denote by $M$ the monomial $M=x_2\cdots x_{n}$ and remark that
	$$NI(G^2):M=(x_1y_1y_2)+(y_jy_{j+1}y_{j+2}y_{j+3}y_{j+4}:1\leq j\leq m-4)+(y_{m-2}y_{m-1}y_m).$$ 
	But $(x_1y_1y_2)+(y_jy_{j+1}y_{j+2}y_{j+3}y_{j+4}:1\leq j\leq m-4)+(y_{m-2}y_{m-1}y_m)$ is the closed neighborhood ideal of the square of path graph on the vertex set $\{x_1,y_1,\ldots,y_m\}$. By \cite[Proposition 1.2]{OO}, we have $\height NI(G^2):M=\left\lceil\frac{m+1}{5}\right\rceil$ therefore 
	$$\height NI(G^2)\leq \height NI(G^2):M=\left\lceil\frac{m+1}{5}\right\rceil.$$
	
	On the other hand, assume that there exist a minimal prime ideal $\frak q$ of $I$ with $\height \frak q<\left\lceil\frac{m+1}{5}\right\rceil$. 
	
	If $m+1=5k$, that is we assume that $\height \frak q<k$, then $\frak q$ must contain a variable from the support of $x_1\cdots x_ny_1y_2$, a variable from the support of $y_{m-2}y_{m-1}y_{m}$ and variables from the monomials with support in $\{y_3,\ldots y_{m-3}\}$. Since $$|\{y_3,\ldots,y_{m-3}\}|=m-5=5(k-2)+4,$$
	we split the set in $k-2$ disjoint sets of cardinality $5$ and one of cardinality $4$, hence at least $k-2$ variables from the sets lie in $\frak q$. Therefore $\height \frak q\geq 1+(k-2)+1=k=\left\lceil\frac{m+1}{5}\right\rceil$, contradiction.
	
	If $m+1=5k+r$, with $1\leq r\leq 4$, that is we assume that $\height \frak q<k+1$, we obtain that
	$$|\{y_3,\ldots,y_{m-3}\}|=m-5=5(k-1)+r-1.$$
	We split the set in $k-1$ disjoint sets with five elements and one with the rest of variables (can be the empty set for $r=1$), hence at least $k-1$ variables belong to $\frak q$. Therefore $\height \frak q\geq 1+(k-1)+1=k+1=\left\lceil\frac{m+1}{5}\right\rceil$, contradiction.
	
	
	It follows that any minimal prime ideal of $NI(G^2)$ has height at least $\left\lceil\frac{m+1}{5}\right\rceil$.
	It results that $\height NI(G^2)=\left\lceil\frac{m+1}{5}\right\rceil.$
\end{proof}

Knowing the height and the projective dimension we can characterize the Cohen-Macaulay property. 
\begin{Corollary}\label{CM broom}
	Let $NI(G^2)\subset S=K[V(G)]$ be the closed neighborhood ideal of the square of broom graph $G=B(n,m)$ on the set of vertices $V(G)=\{x_1,\ldots,x_n,y_1,\ldots,y_m\}$. Then $NI(G^2)$ is Cohen-Macaulay if and only if $m\in\{2,5\}$.
\end{Corollary}

\begin{proof}
	Let $m+1=5k+r$ and $m-2=6p''+d''$, where $0\leq r\leq 4$ and $0\leq d''\leq 5$, that is $5k=6p''+d''-r+3$. We have to consider the following cases:
	\begin{enumerate}
		\item If $r=0$, then $\height(NI(G^2))=k$. We have $5k=6p''+d''+3$ and:
		\begin{itemize}
			\item if $d''=0$, then $\projdim S/NI(G^2)=2p''+1$. In this case, $NI(G^2)$ is Cohen-Macaulay if and only if $k=2p''+1$. Since $5k=6p''+3$, we obtain that $4p''=-2$, contradiction.
			\item if $d''\in\{1,2,3\}$, then $\projdim S/NI(G^2)=2p''+2$. In this case, $NI(G^2)$ is Cohen-Macaulay if and only if $k=2p''+2$. Since $5k=6p''+d''+3$, we obtain that $4p''+7-d''=0$, contradiction.
			\item if $d''\in\{4,5\}$, then $\projdim S/NI(G^2)=2p''+3$. In this case, $NI(G^2)$ is Cohen-Macaulay if and only if $k=2p''+3$. Since $5k=6p''+d''+3$, we obtain that $4p''+12-d''=0$, contradiction.
		\end{itemize}
		\item If $r\geq 1$, then $\height(NI(G^2))=k+1$. We have $5k=6p''+d''-r+3$ and:
		\begin{itemize}
			\item if $d''=0$, then $\projdim S/NI(G^2)=2p''+1$. In this case, $NI(G^2)$ is Cohen-Macaulay if and only if $k=2p''$. Using $5k=6p''-r+3$, we obtain that $4p''+r-3=0$. The only case when we do not get a contradiction is for $r=3$, that is $k=p''=0$ and $m=2$.
			\item if $d''\in\{1,2,3\}$, then $\projdim S/NI(G^2)=2p''+2$. In this case, $NI(G^2)$ is Cohen-Macaulay if and only if $k=2p''+1$. Since $5k=6p''+d''-r+3$, we obtain that $4p''=d''-r-2$. For $r=1$ and $d''=3$ we obtain $p''=0$, that is $m=5$. For all the other values, we obtain a contradiction.
			\item if $d''\in\{4,5\}$, then $\projdim S/NI(G^2)=2p''+3$. In this case, $NI(G^2)$ is Cohen-Macaulay if and only if $k=2p''+2$. Since $5k=6p''+d''-r+3$, we obtain that $4p''=d''-r-7<0$, a contradiction.
		\end{itemize}
	\end{enumerate}
	It results that $NI(G^2)$ is Cohen-Macaulay if and only if $m\in\{2,5\}$.
\end{proof}

\section{The closed neighborhood ideal of the square of double broom graph}
Let $G_1=B(n,m,s)$ be the broom graph, with $n,m,s\geq 2$, and
$$NI(G_1^2)=(x_1\cdots x_ny_1y_2,z_1\cdots z_sy_{m-1}y_m)+(y_jy_{j+1}y_{j+2}y_{j+3}y_{j+4}:1\leq j\leq m-4)$$ 
the closed neighborhood ideal of the square of double broom graph $G_1^2$, $NI(G_1^2)\subset S=K[x_1,\ldots,x_n,y_1,\ldots,y_m,z_1\ldots,z_s]$.

We remark that the closed neighborhood ideal of the square of double broom graph $G_1^2$ can be written as
$$NI(G_1^2)=(x_1\cdots x_ny_1y_2,z_1\cdots z_sy_{m-1}y_m)+I_5(P_{m}),$$
if $m\geq 5$ and we have the following particular cases:
\begin{itemize}
	\item If $m=2$, then $NI(G_1^2)=(x_1\cdots x_ny_1y_2,z_1\cdots z_sy_1y_2)=y_1y_2(x_1\cdots x_n,z_1\cdots z_s)$;
	\item If $m=3$, then $NI(G_1^2)=(x_1\cdots x_ny_1y_2,z_1\cdots z_sy_2y_3)=y_2(x_1\cdots x_ny_1,z_1\cdots z_sy_3)$;
	\item If $m=4$, then $NI(G_1^2)=(x_1\cdots x_ny_1y_2,z_1\cdots z_sy_3y_4)$.
\end{itemize}

Therefore, we start by computing the invariants for these particular cases:

\begin{Proposition}\label{particular cases double broom}
	Let $G_1=B(n,m,s)$ be the double broom graph, with $n,s\geq 2$. Then:
	\begin{itemize}
		\item If $m=2$, then $\height (NI(G_1^2))=1$, $\projdim (S/NI(G_1^2))=2$ and $\reg (S/NI(G_1^2))=n+s+2$;
		\item If $m=3$, then $\height (NI(G_1^2))=1$, $\projdim (S/NI(G_1^2))=2$ and $\reg (S/NI(G_1^2))=n+s+3$;	
		\item If $m=4$, then $\height (NI(G_1^2))=2$, $\projdim (S/NI(G_1^2))=2$ and $\reg (S/NI(G_1^2))=n+s+4$.
	\end{itemize}	
\end{Proposition}

\begin{proof}
	It is clear that $(y_2)$ is a minimal prime ideal for $2\leq m\leq 3$, while for $m=4$ we consider the minimal prime ideal $(y_2,y_3)$. Therefore $\height (NI(G_1^2))=1$, for $2\leq m\leq 3$, and $\height (NI(G_1^2))=2$ for $m=4$.
	
	We also remark that $NI(G_1^2)$ is a complete intersection ideal, hence we have:
	\begin{itemize}
		\item for $m=2$ it results that $$\reg (S/NI(G_1^2))=\reg (NI(G_1^2))-1=2+[\deg(x_1\cdots x_n)+\deg(z_1\cdots z_s)+2-1]-1=n+s+2;$$
		\item for $m=3$ we obtain that $$\reg (S/NI(G_1^2))=1+[\deg(x_1\cdots x_ny_1)+\deg(z_1\cdots z_sy_3)+2-1]-1=n+s+3;$$
		\item for $m=4$ we get that $$\reg (S/NI(G_1^2))=[\deg(x_1\cdots x_ny_1y_2)+\deg(z_1\cdots z_sy_3y_4)+2-1]-1=n+s+4.$$
	\end{itemize} 
	Using Singular \cite{DGPS} we obtain the projective dimension and also we may check the formulae of the Castelnuovo--Mumford regularity.
\end{proof}

For $m\geq 5$ we have that
$$NI(G_1^2)=(x_1\cdots x_ny_1y_2,z_1\cdots z_sy_{m-1}y_m)+I_5(P_{m}).$$

As in the previous section, we will use the mapping cone technique.

Firstly, from the following exact sequences:
\tiny{$$0\rightarrow \frac{S}{(I_5(P_{m}),z_1\cdots z_sy_{m-1}y_m):(x_1\cdots x_ny_1y_2)}(-n-2)\rightarrow \frac{S}{(I_5(P_{m}),z_1\cdots z_sy_{m-1}y_m)}\rightarrow \frac{S}{NI(G_1^2)}\rightarrow 0$$}

\normalsize{we obtain that}
\tiny{$$\reg\frac{S}{NI(G^2)}=\max\left\{\reg \frac{S}{(I_5(P_{m}),z_1\cdots z_sy_{m-1}y_m):(x_1\cdots x_ny_1y_2)}+n+1,\reg \frac{S}{(I_5(P_{m}),z_1\cdots z_sy_{m-1}y_m)}\right\}$$}
\normalsize{and} 
\tiny{$$\projdim\frac{S}{NI(G^2)}=\max\left\{\projdim \frac{S}{(I_5(P_{m}),z_1\cdots z_sy_{m-1}y_m):(x_1\cdots x_ny_1y_2)}+1,\projdim \frac{S}{(I_5(P_{m}),z_1\cdots z_sy_{m-1}y_m)}\right\}$$}
\normalsize Secondly, we use the same method to the following exact sequence
$$0\rightarrow \frac{S}{I_5(P_{m}):(z_1\cdots z_sy_{m-1}y_{m})}(-s-2)\rightarrow \frac{S}{I_5(P_{m})}\rightarrow \frac{S}{(I_5(P_{m}),z_1\cdots z_sy_{m-1}y_m)}\rightarrow 0\qquad (2)$$
and obtain
$$\reg\frac{S}{(I_5(P_{m}),z_1\cdots z_sy_{m-1}y_m)}=\max\left\{\reg \frac{S}{I_5(P_{m}):(z_1\cdots z_sy_{m-1}y_m)}+s+1,\reg \frac{S}{I_5(P_{m})}\right\}$$
and 
$$\projdim\frac{S}{(I_5(P_{m}),z_1\cdots z_sy_{m-1}y_m)}=\max\left\{\projdim \frac{S}{I_5(P_{m}):(z_1\cdots z_sy_{m-1}y_m)}+1,\projdim \frac{S}{I_5(P_{m})}\right\}$$

We start by computing the mentioned invariants for $\frac{S}{I_5(P_{m}):(z_1\cdots z_sy_{m-1}y_m)}$:
\begin{Lemma}\label{reg+pd i_5:m}
	Let $m\geq 5$ be such that $m-2=6p''+d''$ with $p''\geq 0$ and $0\leq d''\leq 5$. Then:
	\begin{itemize}
		\item for $p''\geq 1$ we have that
		$$\reg\left(\frac{S}{I_5(P_{m}):(z_1\cdots z_sy_{m-1}y_m)}\right)=\left\{\begin{matrix}
			4p''-2,&d''=0\\
			4p'',&d''\in\{1,2,3\}\\
			4p''+2,&d''\in\{4,5\}
		\end{matrix}\right.$$
		and $$\projdim\left(\frac{S}{I_5(P_{m}):(z_1\cdots z_sy_{m-1}y_m)}\right)=\left\{\begin{matrix}
			2p''-1&d''=0\\
			2p'',&d''\in\{1,2,3\}\\
			2p''+1,&d''\in\{4,5\}
		\end{matrix}\right.$$
		\item if $p''=0$ and $3\leq d''\leq 5$, then 
		$$\reg\left(\frac{S}{I_5(P_{m}):(z_1\cdots z_sy_{m-1}y_m)}\right)=2 \mbox{ and }\projdim\left(\frac{S}{I_5(P_{m}):(z_1\cdots z_sy_{m-1}y_m)}\right)=1.$$
	\end{itemize}
\end{Lemma}

\begin{proof}
	Assume that $p''\geq 1$. By simple computations, one may note that
	$$I_5(P_m):(z_1\cdots z_sy_{m-1}y_m)=(y_jy_{j+1}y_{j+2}y_{j+3}y_{j+4}:1\leq j\leq m-4):(y_{m-1}y_{m})=$$
	$$=I_5(P_{m-3})+(y_{m-4}y_{m-3}y_{m-2}).$$
	We relabel the variables $y_i$ to $\alpha_{m-i+1}$, $1\leq i\leq m$ and we obtain that 
	$$I_5(P_m):(z_1\cdots z_sy_{m-1}y_m)=I_5(P_{m-3})+(y_{m-4}y_{m-3}y_{m-2})=(\alpha_3\alpha_4\alpha_5)+I_5(P_{\left\{4,\ldots,m\right\}})$$
	Moreover, by Lemma \ref{reg1} \cite[Lemma 1.10]{OO}, for $m-2=6p''+d''$, with $p''\geq 1$, one has
	$$\reg\left(\frac{K[\alpha_3,\ldots,\alpha_m]}{I_5(P_{\{4,\ldots,m\}})+(\alpha_3\alpha_4\alpha_5)}\right)=\left\{\begin{matrix}
		4p''-2,&d''=0\\
		4p'',&d''\in\{1,2,3\}\\
		4p''+2,&d''\in\{4,5\}
	\end{matrix}\right.$$
	and $$\projdim\left(\frac{K[\alpha_3,\ldots,\alpha_m]}{I_5(P_{\{4,\ldots,m\}})+(\alpha_3\alpha_4\alpha_5)}\right)=\left\{\begin{matrix}
		2p''-1&d''=0\\
		2p'',&d''\in\{1,2,3\}\\
		2p''+1,&d''\in\{4,5\}
	\end{matrix}\right..$$
	Therefore, for $m-2=6p''+d''$, with $p''\geq 1$, we obtain
	$$\reg\left(\frac{S}{I_5(P_{m}):(z_1\cdots z_sy_{m-1}y_m)}\right)=\left\{\begin{matrix}
		4p''-2,&d''=0\\
		4p'',&d''\in\{1,2,3\}\\
		4p''+2,&d''\in\{4,5\}
	\end{matrix}\right.$$
	and $$\projdim\left(\frac{S}{I_5(P_{m}):(z_1\cdots z_sy_{m-1}y_m)}\right)=\left\{\begin{matrix}
		2p''-1&d''=0\\
		2p'',&d''\in\{1,2,3\}\\
		2p''+1,&d''\in\{4,5\}
	\end{matrix}\right..$$
	
	For the remaining case $p''=0$ we remark that:
	\begin{itemize}
		\item if $d''=3$, then $m=5$ and 
		$$I_5(P_m):(z_1\cdots z_sy_{m-1}y_m)=(y_1y_2y_3);$$
		\item if $d''=4$, that is $m=6$ we have 
		$$I_5(P_m):(z_1\cdots z_sy_{m-1}y_m)=(y_2y_3y_4);$$
		\item if $d''=3$, then $m=7$ and 
		$$I_5(P_m):(z_1\cdots z_sy_{m-1}y_m)=(y_3y_4y_5).$$
	\end{itemize}
	It is easy to see that in all these cases	$$\reg\left(\frac{S}{I_5(P_{m}):(z_1\cdots z_sy_{m-1}y_m)}\right)=2 \mbox{ and }\projdim\left(\frac{S}{I_5(P_{m}):(z_1\cdots z_sy_{m-1}y_m)}\right)=1.$$
\end{proof}

For the second ideal that we are interested in we have to mention the following result.
\begin{Remark}\label{reg+pd i_5}\rm As before, by  \cite[Theorem 4.1]{HeT} and \cite[Corollary 4.14]{AF} we have that for $m=6p+d$, with $p\geq 0$ and $0\leq d\leq 5$, the projective dimension of $S/I_5(P_m)$ is given by 
	$$\projdim(S/I_5(P_m))=\left\{\begin{matrix}
		2p,&d\neq 5\\
		2p+1,&d=5
	\end{matrix}\right.$$
	and the regularity is 
	$$\reg(S/I_5(P_m))=\left\{\begin{matrix}
		4p,&d\neq 5\\
		4(p+1),&d=5
	\end{matrix}\right..$$
	If we consider $m-2=6p''+d''$ we get that 
	$$\projdim(S/I_5(P_m))=\left\{\begin{matrix}
		2p'',&d''\in\{0,1,2\}\\
		2p''+1,&d''=3\\
		2p''+2,&d''\in\{4,5\}
	\end{matrix}\right.$$
	and the regularity is 
	$$\reg(S/I_5(P_m))=\left\{\begin{matrix}
		4p'',&d''\in\{0,1,2\}\\
		4(p''+1),&d''\in\{3,4,5\}
	\end{matrix}\right..$$
\end{Remark}

We return to the second exact sequence and by mapping cone technique we obtain the followings:

\begin{Proposition}\label{reg+pd i_5,m}
	Let $m\geq 5$ be a natural number such that $m-2=6p''+d''$, with $0\leq d''\leq 5$ and $p''\geq 0$.	The Castelnuovo--Mumford regularity and the projective dimension for $\frac{S}{(I_5(P_{m}),z_1\cdots z_sy_{m-1}y_m)}$ are given by the following formulae
	\begin{itemize}
		\item for $p''\geq 1$
		$$\reg\left(\frac{S}{(I_5(P_{m}),z_1\cdots z_sy_{m-1}y_m)}\right)=\left\{\begin{matrix}
			4p''+s-1,&d''=0\\
			4p''+s+1,&d''\in\{1,2,3\}\\
			4p''+s+3,&d''\in\{4,5\}
		\end{matrix}\right.,s\geq 3$$
		and $$\projdim\left(\frac{S}{(I_5(P_{m}),z_1\cdots z_sy_{m-1}y_m)}\right)=\left\{\begin{matrix}
			2p''&d''=0\\
			2p''+1,&d''\in\{1,2,3\}\\
			2p''+2,&d''\in\{4,5\}
		\end{matrix}\right..$$
		For $s=2$ we obtain 
		$$\reg\left(\frac{S}{(I_5(P_{m}),z_1\cdots z_sy_{m-1}y_m)}\right)=\left\{\begin{matrix}
			4p''+1,&d''=0\\
			4p''+2,&d''\in\{1,2\}\\
			4p''+4,&d''=3\\
			4p''+5,&d''\in\{4,5\}
		\end{matrix}\right..$$
		\item for $p''=0$ and $3\leq d''\leq 5$ we have
		$$\reg\left(\frac{S}{I_5(P_{m}),z_1\cdots z_sy_{m-1}y_m}\right)=s+3 \mbox{ and }\projdim\left(\frac{S}{I_5(P_{m}),z_1\cdots z_sy_{m-1}y_m}\right)=2.$$
	\end{itemize}
\end{Proposition}

\begin{proof}
	The proof follows by
	$$\reg\frac{S}{(I_5(P_{m}),z_1\cdots z_sy_{m-1}y_m)}=\max\left\{\reg \frac{S}{I_5(P_{m}):(z_1\cdots z_sy_{m-1}y_m)}+s+1,\reg \frac{S}{I_5(P_{m})}\right\}$$
	and 
	$$\projdim\frac{S}{(I_5(P_{m}),z_1\cdots z_sy_{m-1}y_m)}=\max\left\{\projdim \frac{S}{I_5(P_{m}):(z_1\cdots z_sy_{m-1}y_m)}+1,\projdim \frac{S}{I_5(P_{m})}\right\}$$
	and using Lemma \ref {reg+pd i_5:m} and Remark \ref{reg+pd i_5}.
\end{proof}

We now focus to the ideals from the first exact sequence. Firstly, we note that:
\begin{Lemma}\label{final ex sq}
	Using the above notations, one has that
	$$(I_5(P_{m}),z_1\cdots z_sy_{m-1}y_m):(x_1\cdots x_ny_1y_2)=NI(B^2(\{z_1,\ldots z_s\},\{y_3,\ldots,y_m\}))$$
\end{Lemma}

\begin{proof}
	It is easy to see that
	$$(I_5(P_{m}),z_1\cdots z_sy_{m-1}y_m):(x_1\cdots x_ny_1y_2)=I_5(P_{m}):(y_1y_2)+(z_1\cdots z_sy_{m-1}y_m)=$$
	$$=(y_3y_4y_5,y_4y_5y_6y_7y_8,\ldots,y_{m-4}y_{m-3}y_{m-2}y_{m-1}y_m)+(z_1\cdots z_sy_{m-1}y_m)=$$
	$$=NI(B^2(\{z_1,\ldots z_s\},\{y_3,\ldots,y_m\}))$$
	where by $B(\{z_1,\ldots z_s\},\{y_3,\ldots,y_m\})$ we denoted the broom graph on the set of vertices $\{y_3,\ldots,y_mz_1,\ldots,z_s\}$ and the set of edges
	$$\{(y_i,y_i+1):3\leq i\leq m\}\cup \{(y_m,z_j):1\leq j\leq s\}.$$
\end{proof}
Since the Castelnuovo--Mumford regularity and the projective dimension of the closed neighborhood ideal of the square of broom graphs were computed in Theorem \ref{projdim+reg}, we obtain that:

\begin{Remark}\rm\label{final ex sq1}
	By Theorem \ref{projdim+reg} we obtain that for $m-4=6p_1+d_1$, with $p_1\geq 0$ and $0\leq d_1\leq 5$,
	$$\reg\frac{S}{(I_5(P_{m}),z_1\cdots z_sy_{m-1}y_m):(x_1\cdots x_ny_1y_2)}=$$
	$$=\reg \frac{S}{NI(B^2(\{z_1,\ldots z_s\},\{y_3,\ldots,y_m\}))}=\left\{\begin{matrix}
		4p_1+s+1,&d_1\in\{0,1\}\\
		4p_1+s+2,&d_1=2\\
		4p_1+s+3,&d_1\in\{3,4\}\\
		4p_1+s+4,&d_1=5
	\end{matrix}\right.$$
	and $$\projdim\frac{S}{(I_5(P_{m}),z_1\cdots z_sy_{m-1}y_m):(x_1\cdots x_ny_1y_2)}=$$
	$$=\projdim\frac{S}{NI(B^2(\{z_1,\ldots z_s\},\{y_3,\ldots,y_m\}))}=\left\{\begin{matrix}
		2p_1+1&d_1=0\\
		2p_1+2,&d_1\in\{1,2,3\}\\
		2p_1+3,&d_1\in\{4,5\}
	\end{matrix}\right..$$
\end{Remark}

We return to the first exact sequence and we conclude that:
\begin{Theorem}\label{projdim+reg double-broom}
	Let $G_1=B(n,m,s)$ be the double broom graph on the set of vertices $V(G_1)=\{x_1,\ldots,x_n,y_1,\ldots,y_m,z_1,\ldots,z_s\}$ and $NI(G_1^2)\subset S=K[V(G_1)]$ the closed neighborhood ideal of the square of $G_1$. If $m-4=6p_1+d_1$, with $p_1\geq 0$ and $0\leq d_1\leq 5$, then:
	$$\reg \frac{S}{NI(G_1^2)}=\left\{\begin{matrix}
		4p_1+s+n+2,&d_1\in\{0,1\}\\
		4p_1+s+n+3,&d_1=2\\
		4p_1+s+n+4,&d_1\in\{3,4\}\\
		4p_1+s+n+5,&d_1=5
	\end{matrix}\right.$$
	and $$\projdim\frac{S}{NI(G_1^2)}=\left\{\begin{matrix}
		2p_1+2&d_1=0\\
		2p_1+3,&d_1\in\{1,2,3\}\\
		2p_1+4,&d_1\in\{4,5\}
	\end{matrix}\right..$$
\end{Theorem}

\begin{proof}
	By the mapping cone we have that 
	$$\reg\frac{S}{NI(G^2)}=$$
	$$=\max\left\{\reg \frac{S}{(I_5(P_{m}),z_1\cdots z_sy_{m-1}y_m):(x_1\cdots x_ny_1y_2)}+n+1,\reg \frac{S}{(I_5(P_{m}),z_1\cdots z_sy_{m-1}y_m)}\right\}$$
	and 
	$$\projdim\frac{S}{NI(G^2)}=$$
	$$=\max\left\{\projdim \frac{S}{(I_5(P_{m}),z_1\cdots z_sy_{m-1}y_m):(x_1\cdots x_ny_1y_2)}+1,\projdim \frac{S}{(I_5(P_{m}),z_1\cdots z_sy_{m-1}y_m)}\right\}.$$
	
	From Remark \ref{final ex sq1} we have that 
	for $m-4=6p_1+d_1$, with $p_1\geq 0$ and $0\leq d_1\leq 5$
	$$\reg \frac{S}{(I_5(P_{m}),z_1\cdots z_sy_{m-1}y_m):(x_1\cdots x_ny_1y_2)}=\left\{\begin{matrix}
		4p_1+s+1,&d_1\in\{0,1\}\\
		4p_1+s+2,&d_1=2\\
		4p_1+s+3,&d_1\in\{3,4\}\\
		4p_1+s+4,&d_1=5
	\end{matrix}\right.$$
	and $$\projdim\frac{S}{(I_5(P_{m}),z_1\cdots z_sy_{m-1}y_m):(x_1\cdots x_ny_1y_2)}=\left\{\begin{matrix}
		2p_1+1&d_1=0\\
		2p_1+2,&d_1\in\{1,2,3\}\\
		2p_1+3,&d_1\in\{4,5\}
	\end{matrix}\right..$$
	
	By Proposition \ref{reg+pd i_5,m} the Castelnuovo--Mumford regularity and the projective dimension for $\frac{S}{(I_5(P_{m}),z_1\cdots z_sy_{m-1}y_m)}$ are given by
	$$\reg\left(\frac{S}{(I_5(P_{m}),z_1\cdots z_sy_{m-1}y_m)}\right)=\left\{\begin{matrix}
		4p''+s-1,&d''=0\\
		4p''+s+1,&d''\in\{1,2,3\}\\
		4p''+s+3,&d''\in\{4,5\}
	\end{matrix}\right.,s\geq 3$$
	and $$\projdim\left(\frac{S}{(I_5(P_{m}),z_1\cdots z_sy_{m-1}y_m)}\right)=\left\{\begin{matrix}
		2p''&d''=0\\
		2p''+1,&d''\in\{1,2,3\}\\
		2p''+2,&d''\in\{4,5\}
	\end{matrix}\right.$$
	where $m-2=6p''+d''$, $p''\geq 1$ and $0\leq d''\leq 5$.
	For $s=2$ we obtain 
	$$\reg\left(\frac{S}{(I_5(P_{m}),z_1\cdots z_sy_{m-1}y_m)}\right)=\left\{\begin{matrix}
		4p''+1,&d''=0\\
		4p''+2,&d''\in\{1,2\}\\
		4p''+4,&d''=3\\
		4p''+5,&d''\in\{4,5\}
	\end{matrix}\right..$$
	For $m\in\{5,6,7\}$ we have 	$$\reg\left(\frac{S}{I_5(P_{m}),z_1\cdots z_sy_{m-1}y_m}\right)=s+3 \mbox{ and }\projdim\left(\frac{S}{I_5(P_{m}),z_1\cdots z_sy_{m-1}y_m}\right)=2.$$
	
	It follows that for $m-4=6p_1+d_1$ with $m\geq 8$ we obtain 
	$$\reg\left(\frac{S}{(I_5(P_{m}),z_1\cdots z_sy_{m-1}y_m)}\right)=\left\{\begin{matrix}
		4p_1+s+1,&d_1\in\{0,1\}\\
		4p_1+s+3,&d_1\in\{2,3,4\}\\
		4p_1+s+5,&d_1=5
	\end{matrix}\right.,s\geq 3$$
	and $$\projdim\left(\frac{S}{(I_5(P_{m}),z_1\cdots z_sy_{m-1}y_m)}\right)=\left\{\begin{matrix}
		2p_1+1&d_1\in\{0,1\}\\
		2p_1+2,&d_1\in\{2,3,4\}\\
		2p_1+3,&d_1=5
	\end{matrix}\right..$$
	For $s=2$ we obtain 
	$$\reg\left(\frac{S}{(I_5(P_{m}),z_1\cdots z_sy_{m-1}y_m)}\right)=\left\{\begin{matrix}
		4p_1+2,&d_1=0\\
		4p_1+4,&d_1=1\\
		4p_1+5,&d_1\in\{2,3,4\}\\
		4p_1+6,&d_1=5
	\end{matrix}\right..$$
	
	Therefore, for $m\geq 8$, we obtain 	$$\reg \frac{S}{NI(G_1^2)}=\left\{\begin{matrix}
		4p_1+s+n+2,&d_1\in\{0,1\}\\
		4p_1+s+n+3,&d_1=2\\
		4p_1+s+n+4,&d_1\in\{3,4\}\\
		4p_1+s+n+5,&d_1=5
	\end{matrix}\right.$$
	and $$\projdim\frac{S}{NI(G_1^2)}=\left\{\begin{matrix}
		2p_1+2&d_1=0\\
		2p_1+3,&d_1\in\{1,2,3\}\\
		2p_1+4,&d_1\in\{4,5\}
	\end{matrix}\right..$$
	
	We still have to consider $m\in\{5,6,7\}$. We have that:
	\begin{itemize}
		\item for $m=5$, that is $m-4=6p_1+d_1$ with $p_1=0$ and $d_1=1$ we get:
		$$\reg\frac{S}{NI(G^2)}=\max\{s+1+n+1,s+3\}=s+n+2$$
		and 
		$$\projdim\frac{S}{NI(G^2)}=\max\{2+1,2\}=3;$$
		\item for $m=6$, that is $p_1=0$ and $d_1=2$ we get:
		$$\reg\frac{S}{NI(G^2)}=\max\{s+2+n+1,s+3\}=s+n+3$$
		and 
		$$\projdim\frac{S}{NI(G^2)}=\max\{2+1,2\}=3;$$
		\item for $m=7$, that is $p_1=0$ and $d_1=3$ we obtain:
		$$\reg\frac{S}{NI(G^2)}=\max\{s+3+n+1,s+3\}=s+n+4$$
		and 
		$$\projdim\frac{S}{NI(G^2)}=\max\{2+1,2\}=3.$$
	\end{itemize}
\end{proof}

Similar to the broom case, we are interested in computing the height of the closed neighborhood ideal of the square of double broom graph.

\begin{Theorem}\label{height double}
	Let $G_1=B(n,m,s)$ the double broom graph on the set of vertices $V(G_1)=\{x_1,\ldots,x_n,y_1,\ldots,y_m,z_1,\ldots,z_s\}$ and $NI(G_1^2)\subset S=K[V(G_1)]$ the closed neighborhood ideal of the square of $G_1$. Then $\height NI(G_1^2)=\left\lceil\frac{m+2}{5}\right\rceil.$
\end{Theorem}

\begin{proof}
	Since 	$$NI(G_1^2)=(x_1\cdots x_ny_1y_2,z_1\cdots z_sy_{m-1}y_m)+I_5(P_{m})$$ we remark that 
	$$NI(G_1^2)\subseteq NI(G_1^2):(z_1\cdots z_{s-1})=$$
	$$=(x_1\cdots x_ny_1y_2)+(y_i\cdots y_{i+4}:1\leq i\leq m-4)+(y_{m-1}y_mz_s)=NI(G^2),$$ where $G$ is the broom graph on the vertex set $V(G)=\{x_1,\ldots,x_n,y_1,\ldots,y_m,z_s\}$. Therefore we have by Theorem \ref{height} that 
	$$\height NI(G_1^2)\leq\height NI(G_1^2):(z_1\cdots z_{s-1})=\height NI(G^2)=\left\lceil\frac{m+2}{5}\right\rceil.$$
	
	On the other hand, assume that there exist a minimal prime ideal $\frak q$ of $I$ with $\height \frak q<\left\lceil\frac{m+2}{5}\right\rceil$.

	If $m+2=5k$, that is we assume that $\height \frak q<k$, then $\frak q$ must contain a variable from the support of $x_1\cdots x_ny_1y_2$, a variable from the support of $z_1\cdots z_sy_{m-1}y_{m}$ and variables from the monomials with support in $\{y_3,\ldots y_{m-3}\}$. Since $$|\{y_3,\ldots,y_{m-3}\}|=m-5=5(k-2)+3,$$
	we split the set in $k-2$ disjoint sets of cardinality $5$ and one of cardinality $3$, hence at least $k-2$ variables from the sets lie in $\frak q$. Therefore $\height \frak q\geq 1+(k-2)+1=k=\left\lceil\frac{m+2}{5}\right\rceil$, contradiction.
	
	If $m+2=5k+r$, with $1\leq r\leq 4$, that is we assume that $\height \frak q<k+1$, we obtain that
	$$|\{y_3,\ldots,y_{m-3}\}|=m-5=5(k-1)+r-2.$$
	We split the set in $k-1$ disjoint sets with five elements and one with the rest of variables (can be the empty set for $r=2$), hence at least $k-1$ variables belong to $\frak q$. Therefore $\height \frak q\geq 1+(k-1)+1=k+1=\left\lceil\frac{m+2}{5}\right\rceil$, contradiction.
	
	
	It follows that any minimal prime ideal of $NI(G_1^2)$ has height at least $\left\lceil\frac{m+2}{5}\right\rceil$.
	It results that $\height NI(G_1^2)=\left\lceil\frac{m+2}{5}\right\rceil.$
\end{proof}

As for the broom case, we characterize the Cohen--Macaulay property:

\begin{Corollary}\label{CM double}
	Let $NI(G_1^2)\subset S=K[V(G_1)]$ be the closed neighborhood ideal of the square of double broom graph $G_1=B(n,m,s)$ on the set of vertices $V(G_1)=\{x_1,\ldots,x_n,y_1,\ldots,y_m,z_1,\ldots,z_s\}$. Then $NI(G_1^2)$ is Cohen-Macaulay if and only if $m=4$.
\end{Corollary}

\begin{proof} 
	Let $m+2=5k+r$ and $m-4=6p_1+d_1$, where $0\leq r\leq 4$ and $0\leq d_1\leq 5$, that is $5k=6p_1+d_1-r+6$. We have to consider the following cases, given by the formulae of the height and of the projective dimension:
	\begin{enumerate}
		\item If $r=0$, then $\height(NI(G_1^2))=k$. We have $5k=6p_1+d_1+6$ and:
		\begin{itemize}
			\item if $d_1=0$, then $\projdim S/NI(G_1^2)=2p_1+2$. In this case, $NI(G_1^2)$ is Cohen-Macaulay if and only if $k=2p_1+2$. Since $5k=6p_1+6$, we obtain that $4p_1=-4$, contradiction.
			\item if $d_1\in\{1,2,3\}$, then $\projdim S/NI(G_1^2)=2p_1+3$. In this case, $NI(G_1^2)$ is Cohen-Macaulay if and only if $k=2p_1+3$. Since $5k=6p_1+d_1+6$, we obtain that $4p_1=d_1-9<0$, contradiction.
			\item if $d_1\in\{4,5\}$, then $\projdim S/NI(G_1^2)=2p_1+4$. In this case, $NI(G_1^2)$ is Cohen-Macaulay if and only if $k=2p_1+4$. Since $5k=6p_1+d_1+6$, we obtain that $4p_1=d_1-14<0$, contradiction.
		\end{itemize}
		\item If $r\geq 1$, then $\height(NI(G^2))=k+1$. We have $5k=6p_1+d_1-r+6$ and:
		\begin{itemize}
			\item if $d_1=0$, then $\projdim S/NI(G_1^2)=2p_1+2$. In this case, $NI(G_1^2)$ is Cohen-Macaulay if and only if $k=2p_1+1$. Using $5k=6p_1-r+6$, we obtain that $4p_1=1-r$. The only case when we do not get a contradiction is for $r=1$, that is $p_1=0$ and $m=4$.
			\item if $d_1\in\{1,2,3\}$, then $\projdim S/NI(G_1^2)=2p_1+3$. In this case, $NI(G_1^2)$ is Cohen-Macaulay if and only if $k=2p_1+2$. Since $5k=6p_1+d_1-r+3$, we obtain that $4p_1=d_1-r-4<0$, a contradiction.
			\item if $d_1\in\{4,5\}$, then $\projdim S/NI(G_1^2)=2p_1+4$. In this case, $NI(G_1^2)$ is Cohen-Macaulay if and only if $k=2p_1+3$. Since $5k=6p_1+d_1-r+6$, we obtain that $4p_1=d_1-r-9<0$, a contradiction.
		\end{itemize}
	\end{enumerate}
	It follows that $NI(G_1^2)$ is Cohen-Macaulay if and only if $m=4$.
\end{proof}

\end{document}